\documentclass[reqno,12pt,usletter]{amsart}
\usepackage{amssymb,amsmath,amsfonts,amsthm}
\begin{document}

\newcommand\bbb{\ensuremath{\mathbb{B}}}
\newcommand\bbc{\ensuremath{\mathbb{C}}}
\newcommand\bbd{\ensuremath{\mathbb{D}}}
\newcommand\bbf{\ensuremath{\mathbb{F}}}
\newcommand\bbk{\ensuremath{\mathbb{K}}}
\newcommand\bbn{\ensuremath{\mathbb{N}}}
\newcommand\bbp{\ensuremath{\mathbb{P}}}
\newcommand\bbq{\ensuremath{\mathbb{Q}}}
\newcommand\bbr{\ensuremath{\mathbb{R}}}
\newcommand\bbt{\ensuremath{\mathbb{T}}}
\newcommand\bbz{\ensuremath{\mathbb{Z}}}

\newcommand\bbbs{\ensuremath{\mathbb{B}\text{ }}}
\newcommand\bbcs{\ensuremath{\mathbb{C}\text{ }}}
\newcommand\bbds{\ensuremath{\mathbb{D}\text{ }}}
\newcommand\bbfs{\ensuremath{\mathbb{F}\text{ }}}
\newcommand\bbks{\ensuremath{\mathbb{K}\text{ }}}
\newcommand\bbns{\ensuremath{\mathbb{N}\text{ }}}
\newcommand\bbps{\ensuremath{\mathbb{P}\text{ }}}
\newcommand\bbqs{\ensuremath{\mathbb{Q}\text{ }}}
\newcommand\bbrs{\ensuremath{\mathbb{R}\text{ }}}
\newcommand\bbts{\ensuremath{\mathbb{T}\text{ }}}
\newcommand\bbzs{\ensuremath{\mathbb{Z}\text{ }}}

\newcommand\mca{\ensuremath{\mathcal{A}}}
\newcommand\mcb{\ensuremath{\mathcal{B}}}
\newcommand\mcc{\ensuremath{\mathcal{C}}}
\newcommand\mcd{\ensuremath{\mathcal{D}}}
\newcommand\mce{\ensuremath{\mathcal{E}}}
\newcommand\mcf{\ensuremath{\mathcal{F}}}
\newcommand\mch{\ensuremath{\mathcal{H}}}
\newcommand\mck{\ensuremath{\mathcal{K}}}
\newcommand\mcl{\ensuremath{\mathcal{L}}}
\newcommand\mcm{\ensuremath{\mathcal{M}}}
\newcommand\mcn{\ensuremath{\mathcal{N}}}
\newcommand\mcp{\ensuremath{\mathcal{P}}}
\newcommand\mcr{\ensuremath{\mathcal{R}}}
\newcommand\mcs{\ensuremath{\mathcal{S}}}
\newcommand\mct{\ensuremath{\mathcal{T}}}
\newcommand\mcu{\ensuremath{\mathcal{U}}}
\newcommand\mcw{\ensuremath{\mathcal{W}}}
\newcommand\mcx{\ensuremath{\mathcal{X}}}
\newcommand\mcy{\ensuremath{\mathcal{Y}}}
\newcommand\mcas{\ensuremath{\mathcal{A}\text{ }}}
\newcommand\mcbs{\ensuremath{\mathcal{B}\text{ }}}
\newcommand\mccs{\ensuremath{\mathcal{C}\text{ }}}
\newcommand\mcds{\ensuremath{\mathcal{D}\text{ }}}
\newcommand\mces{\ensuremath{\mathcal{E}\text{ }}}
\newcommand\mcfs{\ensuremath{\mathcal{F}\text{ }}}
\newcommand\mchs{\ensuremath{\mathcal{H}\text{ }}}
\newcommand\mcks{\ensuremath{\mathcal{K}\text{ }}}
\newcommand\mcls{\ensuremath{\mathcal{L}\text{ }}}
\newcommand\mcms{\ensuremath{\mathcal{M}\text{ }}}
\newcommand\mcns{\ensuremath{\mathcal{N}\text{ }}}
\newcommand\mcps{\ensuremath{\mathcal{P}\text{ }}}
\newcommand\mcrs{\ensuremath{\mathcal{R}\text{ }}}
\newcommand\mcss{\ensuremath{\mathcal{S}\text{ }}}
\newcommand\mcts{\ensuremath{\mathcal{T}\text{ }}}
\newcommand\mcus{\ensuremath{\mathcal{U}\text{ }}}
\newcommand\mcws{\ensuremath{\mathcal{W}\text{ }}}
\newcommand\mcxs{\ensuremath{\mathcal{X}\text{ }}}
\newcommand\mcys{\ensuremath{\mathcal{Y}\text{ }}}

\newcommand\mfa{\ensuremath{\mathfrak{A}}}
\newcommand\mfb{\ensuremath{\mathfrak{B}}}
\newcommand\mfc{\ensuremath{\mathfrak{C}}}
\newcommand\mfg{\ensuremath{\mathfrak{g}}}
\newcommand\mfi{\ensuremath{\mathfrak{I}}}
\newcommand\mfm{\ensuremath{\mathfrak{M}}}
\newcommand\mfs{\ensuremath{\mathfrak{S}}}
\newcommand\mfx{\ensuremath{\mathfrak{X}}}
\newcommand\mfy{\ensuremath{\mathfrak{Y}}}
\newcommand\mfz{\ensuremath{\mathfrak{Z}}}
\newcommand\mfas{\ensuremath{\mathfrak{A}\text{ }}}
\newcommand\mfbs{\ensuremath{\mathfrak{B}\text{ }}}
\newcommand\mfcs{\ensuremath{\mathfrak{C}\text{ }}}
\newcommand\mfgs{\ensuremath{\mathfrak{g}\text{ }}}
\newcommand\mfis{\ensuremath{\mathfrak{I}\text{ }}}
\newcommand\mfms{\ensuremath{\mathfrak{M}\text{ }}}
\newcommand\mfss{\ensuremath{\mathfrak{S}\text{ }}}
\newcommand\mfxs{\ensuremath{\mathfrak{X}\text{ }}}
\newcommand\mfys{\ensuremath{\mathfrak{Y}\text{ }}}
\newcommand\mfzs{\ensuremath{\mathfrak{Z}\text{ }}}

\newcommand\bnd{\ensuremath{\mathcal{B(H)}}}
\newcommand\kpt{\ensuremath{\mathcal{K(H)}}}
\newcommand\fnt{\ensuremath{\mathcal{F(H)}}}
\newcommand\fnts{\ensuremath{\mathcal{F(H)}\text{ }}}
\newcommand\kpts{\ensuremath{\mathcal{K(H)}\text{ }}}
\newcommand\bnds{\ensuremath{\mathcal{B(H)}\text{ }}}

\newcommand\bndx{\ensuremath{\mathcal{B}\mathfrak{(X)}}}
\newcommand\kptx{\ensuremath{\mathcal{K}\mathfrak{(X)}}}
\newcommand\fntx{\ensuremath{\mathcal{F}\mathfrak{(X)}}}
\newcommand\fntxs{\ensuremath{\mathcal{F}\mathfrak{(X)}\text{ }}}
\newcommand\kptxs{\ensuremath{\mathcal{K}\mathfrak{(X)}\text{ }}}
\newcommand\bndxs{\ensuremath{\mathcal{B}\mathfrak{(X)}\text{ }}}

\newcommand\bndxy{\ensuremath{\mathcal{B}\mathfrak{(X,Y)}}}
\newcommand\kptxy{\ensuremath{\mathcal{K}\mathfrak{(X,Y)}}}
\newcommand\fntxy{\ensuremath{\mathcal{F}\mathfrak{(X,Y)}}}
\newcommand\fntxys{\ensuremath{\mathcal{F}\mathfrak{(X,Y)}\text{ }}}
\newcommand\kptxys{\ensuremath{\mathcal{K}\mathfrak{(X,Y)}\text{ }}}
\newcommand\bndxys{\ensuremath{\mathcal{B}\mathfrak{(X,Y)}\text{ }}}

\newcommand\Ran{\text{Ran}}

\newcommand\spr{\text{spr}}
\newcommand\co{\text{co}}

\newcommand\floor[1]{\left\lfloor #1\right\rfloor}
\newcommand\ceil[1]{\left\lceil #1\right\rceil}

\newcommand\ddt[2]{\ensuremath{\frac{d #1}{d #2}}}
\newcommand\pdt[2]{\ensuremath{\frac{\partial #1}{\partial #2}}}

\newcommand\ip[2]{\ensuremath{\left\langle #1,#2\right\rangle}}

\newcommand\half{\ensuremath{\frac{1}{2}}}
\newcommand\spn{\ensuremath{\text{span}}}
\newcommand\fdim{\ensuremath{\text{fdim}}}
\newcommand\fdimC{\ensuremath{\text{rdim}}}
\newcommand\rdim{\ensuremath{\text{rdim}}}
\newcommand{\tnorm}[1]{%
  \left\vert\kern-0.9pt\left\vert\kern-0.9pt\left\vert #1
    \right\vert\kern-0.9pt\right\vert\kern-0.9pt\right\vert}

\newtheorem{theorem}{Theorem}[section]
\newtheorem{lemma}[theorem]{Lemma}
\newtheorem{cor}[theorem]{Corollary}
\newtheorem{propn}[theorem]{Proposition}
\newtheorem{notn}[theorem]{Notation}

\theoremstyle{definition}
\newtheorem{eg}[theorem]{Example}
\newtheorem{defn}[theorem]{Definition}

\title[free products of von Neumann algebras]{The amalgamated free product of semifinite hyperfinite von Neumann algebras over atomic type I subalgebras}

\author[Redelmeier]{Daniel Redelmeier}

\address{Department of Mathematics, Texas A\&M University,
College Station, TX 77843-3368, USA}
\email{delredel@alumni.uwaterloo.ca}
\thanks{Research supported in part by NSF grant DMS--0901220}
\maketitle

\bibliographystyle{hplain}
\begin{abstract}  In this paper we describe the amalgamated free product of finite and semifinite hyperfinite von Neumann algebras over atomic type I subalgebras.  To do this we extend the notions of free dimension and standard embeddings used in the related results for finite von Neumann algebras to ones which work better for the semifinite case.  We also define classes $\mcr_{3}$ (of finite von Neumann algebras) and $\mcr_{4}$ (of semifinite von Neumann algebras) which are closed under such amalgamated free products.
\end{abstract}
\section{Introduction}

Voiculescu's amalgamated free product for C$^{*}$ and von Neumann algebras in \cite{VoicOrig} and \cite{voicAFP} have proved to be very useful constructions.  Early uses include Popa's use of the the amalgamated free product of von Neumann algebras in \cite{popaAFP} to construct subfactors with arbitrary allowable index.

The standard free product of finite hyperfinite von Neumann algebras was described by Dykema in \cite{kenDuke}.  He also described the amalgamated free product of multimatrix algebras in \cite{kenAmJM}.  In \cite{minearxiv} we described the amalgamated free product of finite hyperfinite von Neumann algebras over finite dimensional subalgebras.  Here we will extend this to semifinite hyperfinite von Neumann algebras over atomic type I subalgebras, for which the induced trace is also semifinite.

In \cite{kenLMS}, Dykema also showed that a certain class of von Neumann algebras (referred to as $\mcr$ in that paper and $\mcr_{1}$ here) was closed under amalgamated free products over finite dimensional subalgebrass.  This was used by Kodiyalam and Sunder in \cite{kodsund} and in the related paper by Guionnet,  Jones, and Shlyakhtenko in \cite{GJS}.  This class was extended to a more natural class  (which we refer to as $\mcr_{2}$) in \cite{minearxiv}.  Here we extend these to classes $\mcr_{3}$ of finite von Neumann algebras and $\mcr_{4}$ of semifinite von Neumann algebras which are each closed under amalgamated free products over type I atomic subalgebras.

\section{Basic Theorems and Definitions}

\subsection{Preliminaries}

Throughout this paper we will use finite and semifinite von Neumann algebras, with specified normal faithful semifinite trace.  In most cases we will also have a specified atomic type I subalgebra $D$.  We will in general assume that there is a trace preserving conditional expectation from our main algebras onto $D$ and that the trace will still be semifinite (or finite) on $D$.  If we refer to an expectation without specifying which one, this is  the one we mean.  Unless specified otherwise we will also be assuming homomorphisms are trace preserving.

With these von Neumann algebras in mind, we will use the amalgamated free product for von Neumann algebras.  Note it is clear from the definition of the trace on $A*_{D}B$, $\tau_{D}\circ E_{D}$, that if the trace on $A$, $B$ and $D$ is semifinite then it will be on $A*_{D}B$ (for the semifinite case use \cite{ueda}.

We will often use the following notation $\overset{p_{1}}{A_{1}}\oplus\overset{p_{2}}{A_{2}}\oplus\dots.$,  where the $p_{i}$ denote the central support of the $A_{i}$ (i.e. the identity in $A_{i}$ as part of the larger algebra, also sometimes referred to as the matrix unit).  In the case of atomic type I factors, we may also use the notation $\underset{t}{A}$, where $t$ denotes the trace of a minimal projection in this factor (clearly this doesn't exist in diffuse algebras, and is not necessarily unique if it is not a factor).

\begin{defn}  We will be dealing primarily with four classes of von Neumann algebras, defined as follows:

\begin{itemize}

\item  Let $\mcr_{1}$ (originally $\mcr$ in \cite{kenLMS}) be the class of finite von Neumann algebras which are the finite direct sum of the following types of algebras:
\begin{enumerate}
\item Matrix Algebras
\item $M_{n}\otimes L^{\infty}([0,1])$
\item Hyperfinite II$_{1}$ factors
\item Interpolated Free Group Factors
\end{enumerate}

\item  Let $\mcr_{2}$ be the class of finite von Neumann algebras which are the direct sum of a finite hyperfinite von Neumann algebra and a finite number of interpolated free group factors (as defined in \cite{minearxiv}).

\item Let $\mcr_{3}$ be the class of finite von Neumann algebras which are the direct sum of a finite hyperfinite von Neumann algebra and a countable number of interpolated free group factors.

\item Let $\mcr_{4}$ be the class of semifinite von Neumann algebras which are countable direct sums of the following types of algebras:
\begin{enumerate}
\item Semifinite hyperfinite von Neumann algebras
\item Interpolated free group factors
\item $F\otimes B(\mch)$ where $F$ is an interpolated free group factor and $\mch$ is a separable Hilbert space.
\end{enumerate}

\end{itemize}
\end{defn}
Note each class strictly contains the previous.

On the class $\mcr_{3}$, Dykema defined the notion of \emph{Free Dimension} in \cite{kenDuke} in the following way.

\begin{defn}  Let $A$ be a finite von Neumann algebra with specified normal faithful tracial state which is of the format 
\[
A=H\oplus\bigoplus_{i\in I}\overset{p_{i}}{L(F_{r_{i}})}\oplus\bigoplus_{j\in J}\underset{t_{j}}{M_{n_{j}}},
\]
where $H$ is a diffuse hyperfinite algebra, the $L(F_{r_{i}})$ are interpolated free group factors and the $M_{n_{j}}$ are matrix algebras. The \emph{free dimension} of $A$ (denote $\fdim(A)$) is equal to
\[
1+\left(\sum_{i\in I}\tau(p_{i})^{2}(r_{i}-1)\right)-\sum_{j\in J}t_{j}^{2}.
\]
\end{defn}

This may not be well defined if the interpolated free group factors turn out to be all isomorphic.  The definition can be made rigourous by defining it on a generating set, and this is what we mean when we use it.  This is not as big a problem as it seems, as in most cases we use the free dimension to determine which interpolated free groups we have, and this would then not matter.



The \emph{standard embedding} was introduced by Dykema in \cite{kenDuke}.
\begin{defn}  Let $A=L(F_{r})$ and $B=L(F_{r'})$ for $r<r'$ and let $\phi:A\to B$ be a unital embedding.  We call $\phi$ a \emph{standard embedding} if we can find a semicircular system $\omega=\{X_{t}\}_{t\in T}$ and a copy of $R$, the hyperfinite II$_{1}$ factor free from $\omega$, and $p_{t}\in R$ so that $B$ is generated by  $R$ and $\{p_{t}X_{t}p_{t}\}_{t\in T}$, and such that if $T'=\{t\in T| p_{t}X_{t}p_{t}\in \phi(A)\}$ then $\phi(A)=vN(R\cup\{p_{t}X_{t}p_{t}\}_{t\in T'})$.
\end{defn}

Standard embeddings have the following properties, proved by Dykema in \cite{kenDuke}:

\begin{enumerate}
\item For $A=L(F_{s})$ and $B=L(F_{s'})$, $s<s'$, then for $\phi:A\to B$ and projection $p\in A$, $\phi$ is standard if and only if $\phi|_{pAp}\to\phi(p)B\phi(p)$ is standard.
\item The inclusion $A\to A*B$ is standard if $A$ is an interpolated free group factor and $B$ is an interpolated free group factor, $L(\bbz)$, or a finite dimensional algebra other than $\bbc$.
\item The composition of standard embeddings is standard.
\item For $A_{n}=L(F_{s_{n}})$, with $s_{n}<s_{n'}$ if $n<n'$, and $\phi_{n}:A_{n}\to A_{n+1}$ a sequence of standard embeddings, then the inductive limit of the $A_{n}$ with the inclusions $\phi_{n}$ is $L_{F_{s}}$ where $s=\lim_{n\to\infty}s_{n}$.
\end{enumerate}

\subsection{Previous Results and Lemmas}

Dykema proved the following as Lemma 4.3 in \cite{kenAmJM}.
\begin{lemma}\label{L:removesummand} Let $A$ and $B$ be von Neumann algebras with subalgebra $D$, and let $\mcm=A*_{D}B$.  Let $p$ be a central projection in $A$, and let $\underline{A}=pD\oplus (1-p)A$ and $\underline{\mcm}=\underline{A}*_{D}B$.  Then the central support of $p$ is the same in $\mcm$ and $\underline{\mcm}$. Furthermore $p\mcm p=vN(p\underline{\mcm}p\cup pA)=p\underline{\mcm}p*_{pD}pA$.
\end{lemma}

Dykema also proved the following, as Lemma 4.2 from \cite{kenAmJM}.

\begin{lemma}\label{L:M2lemma}  Let $\mcn=\overset{r_{0}}{H}\oplus\bigoplus\limits_{i\in I}\overset{r_{i}}{F_{i}}\oplus\bigoplus\limits_{j\in J}\overset{q_{j}}{\underset{t_{j}}{M_{n_{j}}}}$, where $H$ is a diffuse hyperfinite algebra.  Let $p\in \mcn$ be a projection such that $\tau(p)=\half$ and so that neither $p$ nor $1-p$ is minimal and central in \mcn.  Let $D=\underset{1/2}{\overset{p}{\bbc}}\oplus\underset{1/2}{\overset{1-p}{\bbc}}$.  Define a matrix algebra $M_{2}(\bbc)$, where $e_{11}=p$ and $e_{22}=1-p$, and let $v=e_{12}$.  Then
\[
\mcm=\mcn*_{D}M_{2}(\bbc)=F\oplus\bigoplus_{k\in K}\overset{q'_{k}}{\underset{t'_{k}}{M_{m_{k}}}},
\]
where $F$ is either an interpolated free group factor or a diffuse hyperfinite algebra, and 
\[
K=\left\{(j,j')|j,j'\in J, j\leq p, j'\leq 1-p,\frac{t_{j}}{n_{j}}+\frac{t_{j'}}{n_{j'}}>\half\right\}.
\]
For $k=(j,j')\in K$, $n_{k}=2n_{j}n_{j'}$ and $t_{k}'=\frac{n_{k}'}{2}\left(\frac{t_{j}}{n_{j}}+\frac{t_{j'}}{n_{j}}-\half\right)$.  We also know $\fdim (\mcm)=\fdim(\mcn)+\frac{1}{4}$, which then determines $F$ if it is an interpolated free group factor ($F$ is only hyperfinite if $\mcn$ is dimension 4).  Furthermore, the inclusion of $L(F_{s_{i}})=r_{i}\mcn r_{i}\to r_{i}Fr_{i}$ is standard. for all $i\in I$ (noting if $I$ is not empty, then $F$ is an interpolated free group factor).
\end{lemma}

The amalgamated free product of multimatrix algebras were described by Dykema as Theorem 5.1 in \cite{kenAmJM}.
\begin{theorem}\label{T:multmatrix}   Let $A$ and $B$ be multimatrix algebras with subalgebra $D$.  Then $A*_{D}B$ is in $\mcr_{3}$.  If $D$ is finite dimensional it is in $\mcr_{2}$ and the hyperfinite part is type I.  Furthermore $\fdim(A*_{D}B)=\fdim(A)+\fdim(B)-\fdim(D)$.
\end{theorem}

The following results were proved in \cite{minearxiv}

\begin{lemma}\label{gluelemma}

Let $\mcn=(M_{m}\oplus M_{n}\oplus B)*_{D}C$ and $\mcm=(M_{n+m}\oplus B)*_{D}C$, where $B$, $C$ are semifinite von Neumann algebras and $D=\bigoplus_{i=1}^{K}\overset{p_{i}^{D}}{\bbc}$ with $K\in\bbn\cup \{\infty\}$.   $\mcn$ is included in $\mcm$ by including $M_{m}$ and $M_{n}$ as blocks on the diagonal of $M_{n+m}$, and $B$ and $C$ by the identity.  Assume there exists a partial isometry in $\mcn$ between minimal projections in $M_{m}$ and $M_{n}$ (for example if there exists a factor $\mathcal{F}$ with $M_{m}\oplus M_{n}\subseteq \mathcal{F}\subseteq \mcn$).  Then for any minimal projection $p\in M_{m}$ such that $p\leq p_{i}^{D}$ for some $i$, $p\mcn p *L(\bbz)\cong p\mcm p$.

\end{lemma}

\begin{lemma}\label{tensorlemma1}  Let $\mcm=((M_{n}\otimes A)\oplus B)*_{D}C$ and $\mcn=(M_{n}\oplus B)*_{D}C$ for $A$ is a finite von Neumann algebras and $B$ and $C$ are semifinite, and $D=\bigoplus_{i=1}^{K}\overset{p_{i}^{D}}{\bbc}$,$k\in \bbn\cup\{\infty\}$, where $C,B, M_{n}$ have expectations onto $D$ and where $E_{D}^{M_{n}\otimes A}=E_{D}^{M_{n}}\otimes \tau_{A}$.  Let $p$ be a minimal projection in $M_{n}$, with  $p\leq p_{i}^{D}$ for some $i$.  Then $p\mcn p*A\cong p\mcm p$.
\end{lemma}

\begin{theorem}\label{maintheorem}  Let $A$ and $B$ be hyperfinite von Neumann Algebras, with finite dimensional subalgebra $D$.  Then $A*_{D}B$ is in $\mcr_{2}$.  Furthermore $\fdim(A*_{D}B)=\fdim(A)+\fdim(B)-\fdim(D)$.
\end{theorem}

The following was proved for $\mcr_{1}$ in \cite{kenLMS} and $\mcr_{2}$ in \cite{minearxiv}:

\begin{theorem}  Both $\mcr_{1}$ and $\mcr_{2}$ are closed under amalgamated free products over finite dimensional subalgebras.  In these cases $\fdim(A*_{D}B)=\fdim(A)+\fdim(B)-\fdim(D)$.
\end{theorem}

\section{Semifinite Hyperfinite von Neumann Algebras}
\subsection{Regulated Dimension}
Unfortunately free dimension is only defined on finite von Neumann algebras, so here we adapt it in a way that generalizes to semifinite algebras.

For algebras in $\mcr_{3}$, where free dimension is defined, we can define the \emph{regulated dimension} in a fairly straight forward way, as follows:
\[
\fdimC(A)=\tau(I_{A})^{2}(\fdim(A)-1).
\]
This means that for an interpolated free group factor $L(F_{t})$ endowed with a trace normalized to $\lambda$, the regulated dimension is $\lambda^{2}(t-1)$, which is the contribution the interpolated free group factor summand gives in the free dimension formula.  For a matrix algebra whose minimal projections have trace $t$, the regulated dimension is $-t^{2}$.  Finally for any diffuse hyperfinite algebra, the regulated dimension is zero.  Unlike the free dimension, the regulated dimension is additive over direct sums (where the trace is preserved in the embeddings), and so the regulated dimension of any algebra in $\mcr_{3}$ can easily be determined from the above.

\begin{lemma}\label{L:rdimP}  Let $A$ be an algebra in $\mcr_{3}$ and let $p\in A$ be a projection in $A$ with full central support.  Then $\rdim(A)=\rdim(pAp)$.
\end{lemma}
\begin{proof}  First consider an interpolated free group factor $A=L(F_{t})$, whose identity has trace $\lambda$.  Then $\rdim(A)=\lambda^{2}(t-1)$.  Let $\gamma=\tau(p)$, then $pAp$ is isomorphic to $L(F_{1+(t-1)/(\gamma/\lambda)^{2}})$.  Thus $\rdim(pAp)=\gamma^{2}(1+(t-1)/(\gamma/\lambda)^{2}-1)=\lambda^{2}(t-1)$, as desired.

For a matrix algebra, we see that compression does not change the size of the minimal projection (note it is important that $p\in A$), and thus does not change the regulated dimension.  For diffuse hyperfinite algebras the regulated dimension remains zero.

Since the regulated dimension is additive over direct sums, and since $p$ has full central support (thus all factor summands in $A$ are represented in $pAp$), the result follows.
\end{proof}

Clearly this also means that the regulated dimension is invariant over amplification.

Unfortunately unlike the free dimension, the regulated dimension can be negative and does not match the number of generators in an free group factor.  It also may not be well defined in the same way as the free dimension if the interpolated free group factors are isomorphic.

Importantly if $\fdim(A*_{D}B)=\fdim(A)+\fdim(B)-\fdim(D)$ then $\rdim(A*_{D}B)=\rdim(A)+\rdim(B)-\rdim(D)$.


We can use the fact that this is invariant under amplification to extend the definition to $\mcr_{4}$.  Since for any $n$, $\rdim(L(F_{s})\otimes \underset{t}{M_{n}})=t^{2}(s-1)$ it is natural define $\rdim(L(F_{s})\otimes \underset{t}{B(\mch)})=t^{2}(s-1)$.  Note that $\underset{t}{B(\mch)}\otimes L(F_{s})\cong \underset{t'}{B(\mch)}\otimes L(F_{1+t^{2}(s-1)/t'^{2}})$ with a trace preserving isomorphism, however in both cases $\fdimC$ is $t^{2}(s-1)$.  While it is known that all factors of the form $B(\mch)\otimes L(F_{s})$ are isomorphic, there is only a trace preserving isomorphism if the interpolated free group factors are isomorphic.

Similarly for $\underset{t}{B(\mch)}$ we naturally define the regulated dimension to be $-t^{2}$, as it is for any $n$ with $\underset{t}{M_{n}}$.  Of course for any diffuse semifinite hyperfinite algebra we set the regulated dimension to zero.

Using the above and the fact that the regulated dimension is additive over direct sums, we can then define it for $\mcr_{4}$, except when this gives us $\infty-\infty$.

\begin{defn}  Let $A$ be a semifinite von Neumann algebra in $\mcr_{4}$, written as:
\[
A=H\oplus \bigoplus_{i\in I}\overset{p_{i}}{L(F_{s_{i}})}\oplus\bigoplus_{j\in J} (L(F_{s_{j}})\otimes \underset{t_{j}}{B(\mch)})\oplus\bigoplus_{k\in K} \underset{t_{k}}{M_{n_{k}}}\oplus\bigoplus_{\ell\in L}\underset{t_{\ell}}{B(\mch)},
\]
where $H$ is a diffuse hyperfinite algebra.  Then
\[
\rdim(A)=\sum_{i\in I}\tau(p_{i})^{2}(s_{i}-1)+\sum_{j\in J}t_{j}^{2}(s_{j}-1)-\sum_{k\in K}t_{k}^{2}-\sum_{\ell\in L}t_{\ell}^{2},
\] 
if this is defined.
\end{defn}
\begin{notn}To simplify things we use $\mcf_{s}^{t}$ to denote $L(F_{1+s/t^{2}})$ for $0<t<\infty$ equipped with a trace so that $\tau(I_{\mcf_{s}^{t}})=t$ and $\underset{1}{B(\mch)}\otimes L(F_{1+s})$ if $t=\infty$.  We will refer to these in general as semifinite interpolated free group factors.
\end{notn}
Using this we can rewrite the definition of regulated dimension in a simpler way.  For $A=\bigoplus_{i\in I} \mcf_{s_{i}}^{t_{i}}\oplus\bigoplus_{j\in J}\underset{t_{j}}{M_{n_{j}}}$, where $n_{j}$ may be infinity, then $\rdim(A)=\sum_{i\in I}s_{i}-\sum_{j\in J}t_{j}^{2}$.

\begin{eg}  Consider the algebra in $\mcr_{4}$ as follows $\bigoplus\limits_{i=1}^{\infty}\left(\mcf_{1}^{1}\oplus \underset{1}{\bbc}\right)$.  If we compute the regulated dimension we see that $\rdim(A)=\infty-\infty$ and thus it is not well defined.
\end{eg}

\subsection{Substandard Embeddings}
When working with semifinite algebras it makes more sense to extend the concept of the standard embedding, which we do in the following way:

\begin{defn}  We say that a trace preserving embedding $\phi:\mcf_{s}^{t}\to\mcf_{s'}^{t'}$ is \emph{substandard} if for some projection $p\in \mcf_{s}^{t}$ with finite (non-zero) trace, the restriction $\phi|_{p\mcf_{s}^{t}p}:p\mcf_{s}^{t}p\to \phi(p)\mcf_{s'}^{t'}\phi(p)$ is a standard embedding or an isomorphism.
\end{defn}

Note it is clearly equivalent to say this is true for any finite trace (non-zero) projection.  Also note that substandard embeddings give us increasing regulated dimension, but not necessarily increasing free dimension.

The first three properties of the standard embedding listed before follow directly for the substandard embedding.  The final is as follows:

\begin{propn}\label{L:substdsemifinite}   For $\{\mcf_{s_{i}}^{t_{i}}\}_{i=1}^{\infty}$ and substandard embeddings $\phi_{i}:\mcf_{s_{i}}^{t_{i}}\to \mcf_{s_{i+1}}^{t_{i+1}}$ then the inductive limit is $\mcf_{s}^{t}$ where $s=\lim_{i\to\infty}s_{i}$ and $t=\lim_{i\to\infty}t_{i}$, and the induced embedding of $\mcf_{s_{i}}^{t_{i}}\to\mcf_{s}^{t}$ is substandard.
\end{propn}
\begin{proof}
Let $\rho_{i}=\phi_{i}\circ\phi_{i-1}\circ\dots\circ\phi_{1}$.  Choose a projection $p\in \mcf_{s_{1}}^{t_{1}}$ with finite non-zero trace $t'$.

 Then $\rho_{i}|_{p}$ considered as map from $p\mcf_{s_{1}}^{t_{1}}p\to\rho_{i}(p)\mcf_{s_{i+1}}^{t_{i+1}}\rho_{i}(p)$ is standard or an isomorphism.  Furthermore we know $\rho_{i}(p)\mcf_{s_{i+1}}^{t_{i+1}}\rho_{i}(p)\cong\mcf_{s_{i+1}}^{t'}$.   Then we see each $\rho_{i-1}(p)\mcf_{s_{i}}^{t_{i}}\rho_{i-1}(p)$ is generated by $R_{i}$ and $S_{i}$, where $R_{i}$ is a copy of the hyperfinite II$_{1}$ factor, and $S_{i}=\{q_{j}X_{j}q_{j}\}_{j\in J_{i}}$ with $\sum_{j\in J_{i}}\tau(q_{j})^{2}=s_{i}$, where the $X_{j}$ are a semicircular system and $q_{j}\in R_{i}$, and so that $\phi_{i}(R_{i})=R_{i+1}$ and $\phi_{i}(S_{i})\subseteq S_{i+1}$.  Thus $\rho_{i}(p)\mcf_{s_{i+1}}^{t_{i+1}}\rho_{i}(p)\cong\mcf_{s_{i+1}}^{t'}$.

  Then for each $i$, $\mcf_{s_{i}}^{t_{i}}$ is generated by $S_{i}$ and an amplification of $R_{i}$.  The inductive limit is then generated by $S=\cup_{i=1}^{\infty}S_{i}$ and $R$, a hyperfinite II$_{1}$ (if $\lim_{i}t_{i}<\infty$) or II$_{\infty}$ (if $\lim_{i}t_{i}=\infty$) factor.  If we let $J=\cup_{i}J_{i}$ then we see $\sum_{j\in J}\tau(q_{j})^{2}=\lim_{i} s_{i}$  Thus the inductive limit is $\mcf_{s}^{t}$ as desired.



\end{proof}

Finally, for convenience of notation we will further extend the definition of substandard in the following way:

\begin{defn}  For von Neumann algebras $A$ and $B$ in $\mcr_{4}$, an embedding $\phi:A\to B$ is said to be substandard if for every semifinite interpolated free group factor summand of $A$,   $\mcf_{s}^{t}$ with central support $p$, $\phi:\mcf_{s}^{t}\to \phi(p)B\phi(p)$ is substandard.
\end{defn}

\subsection{Main Lemmas}

The following lemma, based on Lemma 5.2 in \cite{kenAmJM} will be used for almost all of our results.

\begin{lemma}\label{L:abelianD}  Let $\mcr$ be a class of von Neumann algebras such that if $A\in\mcr$ and $p\in A$ then $pAp\in \mcr$.  Let $A$ and $B$ be von Neumann algebras in $\mcr$ with atomic type I subalgebra $D$.  If for all $A'$ and $B'$ in $\mcr$ with abelian multimatrix subalgebra $D'$,  $A'*_{D'}B'$ is in $\mcr_{4}$, then so is $A*_{D}B$.  Furthermore conditions on the number of $\mcf_{s}^{t}$ factor summands and type I atomic factor summands are preserved, and if we know $\fdimC(A'*_{D'}B')=\fdimC(A')+\fdimC(B')-\fdimC(D')$ then the same holds for $A$, $B$, and $D$.  Also if $A$ and $B$ are finite, then $A*_{D}B$ is too, and we can then make the same statement about $\mcr_{3}$
\end{lemma}
\begin{proof}  Let $D=\bigoplus_{k\in I_{D}}\overset{p_{k}^{D}}{M_{n_{k}^{D}}}$ (where we allow $n_{k}=\infty$, in which case we mean $B(\mch)$).  For each $k\in I_{D}$, let $\{e_{i,j}^{(k)}\}$ be the standard basis for $M_{n_{k}^{D}}$.  Let $e=\sum_{k\in I_{D}}e_{1,1}^{(k)}$ (with the sum converging in the SOT).  Note that this is a projection in $D$ with central support $I$, and $eDe$ is abelian. Let $A'=eAe, B'=eBe,$ and $D'=eDe$.  Let $\mcm=A*_{D}B$ and let $\mcm'=vN(A'\cup B')$, and since $A'\subseteq A$, $B'\subseteq B$ we see that $A'$ and $B'$ are free with amalgamation over $D'$, and thus $\mcm'=A'*_{D'}B'$. Let $V=\{e_{1,j}^{(k)}|k\in I_{D},1<j\leq n_{k}^{D}\}$.  Note then that $A=vN(A'\cup V)$, $B=vN(B'\cup V)$, and $D=vN(D'\cup V)$.  Then $\mcm=vN(\mcm'\cup V)$. Thus $\mcm'=e\mcm e$.   By our assumption we know $A'$ and $B'$ are in $\mcr$ and $D'$ is an abelian multimatrix algebra.  Thus by assumption $\mcm'$ is of the correct form, which implies $\mcm$ is of the same form with the same number of factor summands of each kind (where the kinds are: diffuse hyperfinite, type I atomic, and $\mcf_{s}^{t}$ as this step may change interpolated free group factors to $\mcf_{t}^{\infty}$ or matrix algebras to $B(\mch)$).  Now if $\fdimC(A'*_{D'}B')=\fdimC(A')+\fdimC(B')-\fdimC(D')$, then since the central support of $e$ is the identity $\fdimC(A*_{D}B)=\fdimC(A)+\fdimC(B)-\fdimC(D)$.

\end{proof}

\begin{lemma}\label{L:infchainlemma}  Let $A$ be a semifinite hyperfinite von Neumann algebra with trace preserving conditional expectation onto a subalgebra $D=\bigoplus\limits_{n\in N}\underset{t_{n}}{\overset{p_{n}}{\bbc}}$, for some countable index set $N$, so that the trace is semifinite on $D$ as well.  Then there exists a chain of multimatrix subalgebras $A_{j}$, so that $D\subseteq A_{1}\subseteq A_{2}\dots$ and $\cup_{j=1}^{\infty}A_{j}$ is dense in $A$.
\end{lemma}
\begin{proof}  
First we assume $N$ is finite a finite set and begin by dividing it into type I and type II parts and approximating them separately. If $A$ is type I and finite, and thus of the form $B\oplus(C\otimes L^{\infty}([0,1]))$, where $B$ and $C$ are multimatrix algebras. We can choose a representation so that every $p_{n}$, $n\in N$, is composed of diagonal matrices in $B$ and $C$ and characteristic functions in $L^{\infty}([0,1])$. Let $C=\oplus_{k\in K}C_{k}$ where each $C_{k}$ is a matrix algebra.  Then for each $k$ in $K$ we can choose a partition $P_{k}$ of $[0,1]$ into a countable number of measurable subsets so that for every $S\in P_{k}$ and $e_{j,j}^{(k)}\in C_{k}$,  $e_{j,j}^{(k)}\otimes\chi_{S}\leq p_{n}$ for some $n\in N$.  Then we set $A_{1}=B\oplus\bigoplus_{k\in K}\left(C_{k}\otimes L^{\infty}(P_{k})\right)$, and note $D\subseteq A_{1}$.  We can then define each $A_{j}$ by further refining $P_{k}$ to get the desired sequence.\\

Next we do the finite type II case (with finite $N$).  To make this easier to extend to the infinite $N$ semifinite case, we will not assume normalised trace, and consider it as a cut-down of a larger algebra.  We consider it as the cutdown of an algebra $A'=L^{\infty}(X,\mu)\otimes R$, with centre-valued trace $\tau_{Z}$ so that $\tau_{Z}(A')$ is the constant function on $X$ of value $2^{K}$ for some $K$.  If we let $p_{n'}=I_{A'}-I_{A}$, we can rephrase this as finding an approximating sequence for $A'$ with subalgebra $D'$ spanned by ${p_{n}}_{n\in N\cup\{n'\}}$, and then cut down the result by $I_{A}$.  So without loss of generality we will assume our algebra is $L^{\infty}(X,\mu)\otimes R$ (essentially we are back to working in the normalised trace case).

As before we pick subalgebras $M_{2^{k+K}}$ of $R$ so that $\cup_{k=-K}^{\infty}M_{2^{k+K}}$ is dense in $R$, and denote the standard basis elements of $M_{2^{k+K}}$ by $e_{i,j}^{(k)}$ and the trace of $e_{i,i}^{(k)}$ is $2^{-k}$, and use the inclusion $M_{2^{k+K}}\to M_{2^{k+1+K}}$ where $e_{i,j}^{(k)}\to e_{2i-1,2j-1}^{(k+1)}+e_{2i,2j}^{(k+1)}$.

Let $f_{n}\in L^{\infty}(X,\mu)$ be the centre-valued trace of $p_{n}$ in $A$, for all $n\in N$.  Then let $S_{n,k}=f^{-1}_{n}\left(\cup_{\ell=0}^{\infty})\left[\frac{1}{2^{k}},\frac{2}{2^{k}}\right)\right)$ for $k\geq -K$.  Then note $f_{n}=\sum_{k=-K}^{\infty}\frac{1}{2^{k}}\chi_{S_{n,k}}$.

Now define $P_{0,k}$ to be a countable partition of $X$ so that for every $n\in N$, each $S_{n,k}$ with non-zero measure is the union of sets in $P_{0,k}$, and so that $P_{0,k}$ refines $P_{0,k-1}$ (note since $N$ is finite and $k$ is bounded below, this can be accomplished with a finite set of intersections).  For each $S\in P_{0,k}$, let $s_{S,k}$ be the number of $n\in N$ such that $S\subseteq S_{n,k}$.  Inductively define $r_{S,k}=s_{S,k}+2r_{S',k-1}$ where $S\subseteq S'\in P_{0,k-1}$, starting with  $P_{0,-K}=\{X\}$ and $r_{X,-K}=0$.

Then define $A_{0}$ to be the multimatrix algebra spanned by the basis
\[
B_{0}=\{e_{i,j}^{(k)}\otimes \chi_{S}|k\geq-K, S\in P_{0,k},2r_{S',k-1}< i,j\leq r_{S,k},S\subseteq S'\in P_{0,k-1}\}.
\]
Now then we can partition the diagonal basis elements, $e_{i,i}^{(k)}\otimes\chi_{S}$, from $B_{0}$ into sets $\{F_{n}\}_{n\in N}$ by placing one element of the form $e_{i,i}^{(k)}\otimes \chi_{S}$ into each $F_{n}$ where $S\subseteq S_{n,k}$, which works because there are exactly $s_{S,k}$ such elements.

This means for any $x\in X$ (except possibly on a set of measure zero an $n\in N$, we see that:
\[
\sum\limits_{\substack{e_{i,i}^{(k)}\otimes \chi_{S}\in F_{n},\\x\in S}}\tau(e_{i,i}\otimes \chi_{S})=\sum_{k=-K}^{\infty}\frac{1}{2^{k}}\chi_{S_{n,k}}=f_{n}(x).
\]
Thus we have a set of orthogonal projections $\{p_{n}'\}_{n\in N}$ with centre-valued trace equal to $\{f_{n}\}_{n\in N}$.  These are then equivalent to the $p_{n}$ in $D$, and thus without loss of generality we can identify them with $D$.

To define $B_{m}$, we first define $P_{m,k}$ to be a refinement of $P_{m-1,k}$ such for any $S\in P_{m,k}$ either $0<\mu(S)\leq 2^{-m}$ or $S$ is a single atom.  We define $r_{S,k}$ and $s_{S,k}$ in the same way (note for $S\in P_{m,k}$ if there exists an $S'\in P_{m',k}$ with $S\subseteq S'$ and $m>m'$ then $r_{S,k}=r_{S',k}$ and $s_{S,k}=s_{S',k}$ so this is well defined, and we do not need to add an index $m$).

Then let $A_{m}$ be the multimatrix algebra spanned by
\begin{multline*}
B_{m}=\{e_{i,j}^{(k)}\otimes \chi_{S}|m< k, S\in P_{m,k},2r_{S',k-1}< i,j\leq r_{S,k},S\subseteq S'\in P_{m,k-1},\\
\text{or if } k=m,S\in P_{m,k},1\leq i,j\leq r_{S,m} \}.
\end{multline*}
Recalling the inclusion of $M_{2^{m-1}}$ into $M_{2^{m}}$ and that each $P_{m,k}$ is a refinement of $P_{m-1,k}$, we see $A_{m-1}\subseteq A_{m}$.

We have established that for any $x\in X$ (except possibly on a set of measure zero) as $m$ goes to infinity the sum of  $\{\tau(e_{i,i}^{(k)})=2^{-k}|e_{i,i}^{(k)}\otimes \chi_{S}\in B_{0},x\in S,k\leq m\}$ goes to $2^{K}$.  Since for $A_{m}$ there is a set $S$ containing $x$ with measure less than $\max\{\mu(\{x\}),2^{-m}\}$ and matrix subalgebra $M_{m'}\otimes \chi_{S}\subseteq M_{2^{m}}\otimes \chi_{S}$, whose trace is $\mu(S)$ times that sum.  Thus this is dense in $\cup_{m} M_{2^{m}}\otimes L^{\infty}(X,\mu)$ and thus dense in $A$\\

If $N$ is only countable (and thus $A$ may be only semifinite), then let $q_{n}$ be the sum of the first $n$ $p_{j}$ in $D$.  Then use the finite case to construct an approximating chain $A_{i,1}$ of $q_{1}Aq_{1}$ which contains $q_{1}Dq_{1}$.  We claim we can use the finite case again to let $A_{i,n}$ be an approximating chain of $q_{n}Aq_{n}$ containing $q_{n}Dq_{n}$ and so that $A_{i,n}\subseteq A_{i,n+1}$ for all $i$ and for $n\geq 1$. 

For the type I part, the above construction naturally allows us to do this.  For the type II case we have to make some minor modifications to make this work.

For the type II construction we only note that in going from $q_{n}Aq_{n}$ to $q_{n+1}Aq_{n+1}$ we can consider it now as a cutdown of some $L^{\infty}(X',\mu')\otimes R\otimes M_{2^{K'}}$ where $X'$ contains $X$, and $\mu'$ restricts to $\mu$ on $X$.  We can then write this as $(L^{\infty}(X,\mu)\otimes R\otimes M_{2^{K'}})\oplus (L^{\infty}(X'\backslash X,\mu')\otimes R\otimes M_{2^{K'}}) $, and approximate the second direct summand separately.  The tensor with $M_{2^{K'}}$ only changes the value $K$, which does not affect the $B_{m}$ for $m>-K$.  We are also adding another projection $p_{n+1}$ (or two if we need a new  $p_{(n+1)'}$) , which creates a new partition $S_{n,K}$, but this merely further refines $P_{m,k}$, meaning with this construction each $B_{m}'$ will contain our original $B_{m}$, thus completing the proving the claim. 

Then define our approximating chain to be $A_{i}=A_{i,i}\oplus (1-q_{i})D(1-q_{i})$, completing the proof.

\end{proof}

\begin{defn}  We call an embedding $\phi:\mcn\to\mcm$ a \emph{simple step} if it follows one of the two following patterns:
\begin{enumerate}
\item $\mcn=\overset{p}{A}\oplus \overset{q}{B}$, $\mcm=\left(\bigoplus\limits_{i=1}^{n}\overset{p_{i}}{A}\right)\oplus \overset{q}{B}$, with $p=\sum_{i=1}^{n}p_{i}$, and $\phi(a,b)=(a,\dots,a,b)$.  We call this a \emph{simple step of the first kind}.
\item $\mcn=\underset{t}{M_{n}}\oplus \underset{t}{M_{m}}\oplus \overset{p}{B}$, $\mcm=\underset{t}{M_{n+m}}\oplus\overset{p}{B}$ where $\phi(x,y,b)=\left(\left[\begin{matrix} x & 0 \\ 0 & y\\\end{matrix}\right],b\right)$.  We call this a \emph{simple step of the second kind}.
\end{enumerate}
\end{defn}

\begin{lemma}\label{L:decomp2}  For two multimatrix algebras $\mcn$ and $\mcm$ and a trace preserving inclusion $\phi:\mcn\to\mcm$, then $\phi$ can be written as a composition of a (possibly countably infinite) sequence of simple steps.
\end{lemma}
\begin{proof}  The proof is identical to Lemma 3.5 in \cite{minearxiv}, following directly from examining either the Bratteli diagrams or the minimal projections.  The only difference is that since we have a countably infinite number of matrix algebras, each of which can be split into a countably infinite number of matrix algebras, we may need a countably infinite number of steps.
\end{proof}

\subsection{Main Theorem and Examples}

\begin{theorem}\label{T:maininftheorem}  
Let $A$ and $B$ hyperfinite von Neumann algebras with trace preserving conditional expectations onto a type I atomic subalgebra $D$ then:
\begin{enumerate}  
\item If $A$ and $B$ are finite then $A*_{D}B$ is in $\mcr_{3}$
\item If $A$ and $B$ are semifinite, and the trace on $D$ is still semifinite, then $A*_{D}B$ is in $\mcr_{4}$.
\end{enumerate}
Furthermore  $\fdimC(A*_{D}B)=\fdimC(A)+\fdimC(B)-\fdimC(D)$ if this is well defined.

\end{theorem}
\begin{proof} By Lemma \ref{L:abelianD}  assume without loss of generality that $D$ is abelian. If $D$ is finite dimensional we can apply Theorem \ref{maintheorem} (since if $A$ and $B$ were not finite, then the trace on the finite dimensional $D$ would not be semifinite),  so we may also assume $D=\bigoplus\limits_{k=1}^{\infty}\overset{p_{k}^{D}}{\underset{t_{k}^{D}}{\bbc}}$.

Use Lemma \ref{L:infchainlemma} to define the sequences $A_{i}$ and $B_{j}$ to be chains of subalgebras of $A$ in $B$, respectively, containing $D$.   Let $q_{k}^{D}=\sum_{k'=1}^{k}p_{k'}^{D}$.  Then let
\[
\mcm(i,j,k)=\left(q_{k}^{D}A_{j}q_{k}^{D}*_{q_{k}^{D}Dq_{k}^{D}}q_{k}^{D}B_{j}q_{k}^{D}\right),
\] and let $\mcn(i,j,k)$ be
\[
\left(q_{k-1}^{D}A_{j}q_{k-1}^{D}*_{q_{k-1}^{D}Dq_{k-1}^{D}}q_{k-1}^{D}B_{j}q_{k-1}^{D}\right)\oplus\left(p_{k}^{D}A_{i}p_{k}^{D}*_{p_{k}^{D}Dp_{k}^{D}}p_{k}^{D}B_{j}p_{k}^{D}\right).
\]


 Note that in the construction from Lemma \ref{L:infchainlemma}, if $A^{(a)}$ is the atomic part of $A$, then each $A_{i}$ contains $q_{k}^{D}A^{(a)}q_{k}^{D}$ for some $k$ (and the same is true for $B$).

For fixed $k$, the sequence $\mcm(i,j,k)$, is an approximating sequence for the amalgamated free product of hyperfinite von Neumann algebras over $q_{k}^{D}D$ (which is finite dimensional), as in Theorem \ref{maintheorem}.  Thus we can choose $i_{k}$ and $j_{k}$ sufficiently large to be the start of the sequence in the proof of that theorem.  In particular, for any minimal projection $p\in A_{i_{k}}$ (resp $B_{j_{k}}$) from the diffuse part of $A$ (resp. $B$), if $p\leq p_{k'}^{D}$ with $k'\leq k$ then $\tau(p)<t_{k'}^{D}-\tau(p')$ for any minimal projection $p'\in B$ (resp $A$) with $p'\leq p_{k'}^{D}$, unless $p_{k'}^{D}$ is minimal in $B$ (resp. $A$), and also if the number of interpolated free group factors in $\mcm(i,j,k)$ is stable for all $i\geq i_{k}$ and $j\geq j_{k}$.  Furthermore we can assume that if $\fdimC(A)$ (resp $B$) is finite then $\fdimC(q_{k}A_{i_{k}}q_{k})>\fdimC(q_{k}Aq_{k})-1/k$.  We chose $i_{k}$ and $j_{k}$ also sufficiently large so that $A_{i}$ contains $q_{k}^{D}A^{(a)}q_{k}^{D}$ and $B_{j}$ contains $q_{k}^{D}B^{(a)}q_{k}^{D}$ for all $i,j\geq i_{k},j_{k}$.

Our goal is to show that the the inclusions of the sequence
\[
\mcm(i_{2},j_{2},1)\to \mcn(i_{2},j_{2},2)\to \mcm(i_{2},j_{2},2)\to \mcm(i_{3},j_{3},2)\to \mcn(i_{3},j_{3},3)\to\dots
\]
are substandard.

First examine the steps of the form $\mcm(i_{k},j_{k},k)\to \mcm(i_{k+1},j_{k+1},k)$.  Since $k$ does not change in this step, it is just like the steps in the proof of Theorem \ref{maintheorem} (proof from \cite{minearxiv}).  Using Lemma \ref{L:decomp2}, we break this up into a sequence of inclusions induced by simple steps in the subalgebras of $A$ or $B$, so we need only show that this type of inclusion is substandard.  Consider the inclusion $\mcm(i,j,k) \to \mcm(i+1,j,k)$ with $i_{k}\leq i<i_{k+1}$ and $j_{k}\leq j\leq j_{k+1}$ where $A_{i}\to A_{i+1}$ is a simple step (note this may not work precisely with our original indexing, but this does not affect the proof and eases notation).

If the inclusion $A_{i}\to A_{i+1}$ is a simple step of the first kind, which makes copies of a summand, $M_{n}$, of $A_{i}$, then either $M_{n}$ is contained in some interpolated free group factor summand, or it is the corner of some matrix factor summand (as in the proof of Theorem \ref{maintheorem}).  In the former case we apply Lemma \ref{tensorlemma1} to show this is a standard embedding.  In the latter we see the matrix factor is included in another matrix factor or factors.

Alternately, if the inclusion $A_{i}\to A_{i+1}$ is a simple step of the second kind, then we are adding a partial isometry $v$ connecting two matrix algebras in $A_{i}$.  Again we proceed exactly as in Theorem \ref{maintheorem}.  If both $vv^{*}$ and $v^{*}v$ are in an interpolated free group factor we apply Lemma \ref{gluelemma} to show this is a standard embedding.  If only one of $vv^{*}$ and $v^{*}v$ is in an interpolated free group factor we apply Lemma \ref{L:M2lemma} to show that this is a substandard embedding. If neither $vv^{*}$ nor $v^{*}v$ are in interpolated free group factors, then they are both in matrix algebras, then these are embedded in another matrix algebra.

Thus applying Proposition \ref{L:substdsemifinite} to this sequence of inclusions we see then that any interpolated free group factor summand in $\mcm(i_{k-1},j_{k-1},k)$, is embedded in another in $ \mcm(i_{k},j_{k},k)$ by a substandard embedding, and so this inclusion is substandard.\\

Next we examine the inclusions of the form $\mcn(i_{k},j_{k},k)\to \mcm(i_{k},j_{k},k)$.  This is similar to the proof of Theorem \ref{T:multmatrix} (from \cite{kenAmJM}), since $i_{k}$ and $j_{k}$ do not change, and thus we have a fixed pair of multimatrix algebras $A_{i_{k}}$ and $B_{j_{k}}$.

Let $q_{k}^{D}A_{i_{k}}q_{k}^{D}=\bigoplus_{\ell\in I_{A_{i_{k}}}}\overset{p_{\ell}}{M_{n_{\ell}}}$ and $q_{k}^{D}B_{j_{k}}q_{k}^{D}=\bigoplus_{\ell\in I_{B_{i_{k}}}}\overset{p_{\ell}}{M_{n_{\ell}}}$.  Now for each $\ell\in I_{A_{i_{k}}}\cup I_{B_{j_{k}}}$, if $p_{\ell}$ is not orthogonal to either $q_{k-1}^{D}$ or $p_{k}^{D}$, then choose a $v_{\ell}\in M_{n_{\ell}}$ which is a partial isometry so that $v_{\ell}v_{\ell}^{*}\leq p_{k}^{D}$ and $v_{\ell}^{*}v_{\ell}\leq q_{k-1}$, and each is minimal in $A_{i_{k}}$ or $B_{j_{k}}$.  Let $V$ be the set of these $v_{\ell}$.  Then note that $\mcn(i_{k},j_{k},k)\cup V$ generates $\mcm(i_{k},j_{k},k)$.  Let $\mcn(i_{k},j_{k},k,V')=vN(\mcn(i_{k},j_{k},k)\cup V')$ for $V'\subseteq V$.  We show that for $v\in V$ and $V'\subseteq V\backslash\{v\}$, the inclusion $\mcn(i_{k},j_{k},k,V')\to \mcn(i_{k},j_{k},k,V'\cup\{v\})$ is substandard.

Let $q=vv^{*}+v^{*}v$, and examine $q\mcn(i_{k},j_{k},k,V')q$, and the algebra $M_{2}$ generated by $v$.  As in the proofs of Theorem \ref{T:multmatrix} and Theorem \ref{maintheorem}, we see that 
\[
q\mcn(i_{k},j_{k},k,V')q*_{\overset{vv^{*}}{\bbc}\oplus\overset{v^{*}v}{\bbc}}M_{2}=q\mcn(i_{k},j_{k},k,V'\cup \{v\})q.
\]

If either of $vv^{*}$ or $v^{*}v$ are minimal and central in $q\mcn(i_{k},j_{k},k,V')q$, (without loss of generality assume $vv^{*}$), then $vv^{*}$ is a minimal projection in a matrix algebra summand $M_{n}$ in $\mcn(i_{k},j_{k},k,V')$.  Any matrix algebra summand  $M_{m'}$ under $v^{*}v$ in  $q\mcn(i_{k},j_{k},k,V')q$ must be part of a matrix algebra summand $M_{m}$ in  $\mcn(i_{k},j_{k},k,V')$. Then the inclusion of $M_{m}$ in $\mcn(i_{k},j_{k},k,V'\cup\{v\})$ maps it to the corner of a matrix summand $M_{m+nm'}$.  Any interpolated free group factor under $v^{*}v$, is included into an amplification of itself, which is a substandard embedding.  Finally if there is a diffuse hyperfinite algebra under $v^{*}v$ then it is included in a diffuse hyperfinite algebra.

If neither $vv^{*}$ nor $v^{*}v$ is minimal and central, then we can apply Lemma \ref{L:M2lemma}.  This shows us that the inclusion of $q\mcn(i_{k},j_{k},k,V')q$ into $q\mcn(i_{k},j_{k},k,V'\cup \{v\})q$ is substandard, and thus so is the inclusion of  $\mcn(i_{k},j_{k},k,V')$ into $\mcn(i_{k},j_{k},k,V'\cup \{v\})$. Thus, applying Proposition \ref{L:substdsemifinite}, so is the inclusion $\mcn(i_{k},j_{k},k)\to \mcm(i_{k},j_{k},k)$.\\

In the final case, $\mcm(i_{k},j_{k},k-1)\to \mcn(i_{k},j_{k},k)$ is immediate.

Thus the inclusions in the chain,
\[
\mcm(i_{2},j_{2},1)\to \mcn(i_{2},j_{2},2)\to \mcm(i_{2},j_{2},2)\to \mcm(i_{3},j_{3},2)\to \mcn(i_{3},j_{3},3)\to\dots,
\]
are all substandard.  At all stages this is the direct sum of a finite number of interpolated free group factors and a hyperfinite algebra.  Note at each stage each interpolated free group factor is mapped into another interpolated free group factor with both trace and regulated dimension at least as large.  Thus, applying Proposition \ref{L:substdsemifinite} once more, the inductive limit is a countable direct sum of factors of the form $\mcf_{s}^{t}$ and a hyperfinite algebra, and thus in $\mcr_{4}$.  If $A$ and $B$ are finite, then clearly $A*_{D}B$ is finite, and thus in $\mcr_{3}$.\\

If all the regulated dimensions are defined, as well as the sum then it is clear that $\fdimC (A*_{D}B)=\lim_{k\to\infty}\fdimC(\mcm(i_{k},j_{k},k))$.  Choose $k\in I_{D}$ and denote $A_{i_{k}}=\bigoplus_{\ell\in I_{A}}\overset{p_{\ell}}{\underset{t_{\ell}}{M_{n_{\ell}}}}$.
Then 
\[
\fdimC(q_{k}A_{i_{k}}q_{k})=-\sum_{\ell\in I_{A},p_{\ell}q_{k}\ne 0}t_{\ell}^{2}
\]
and 
\[
\fdimC(A_{i_{k}})=-\sum_{\ell\in I_{A}}t_{\ell}^{2}.
\]
For every $\ell\in I_{D}$ so that $p_{\ell}q_{k}=0$, there is a minimal projection $p_{\ell}'\in M_{n_{\ell}}$ less than some $p_{k'}^{D}$ with $k'>k$.  This tells us that
\[
\sum_{\ell\in I_{D},p_{\ell}q_{k}=0}t_{\ell}^{2}<\sum_{k'>k}(t_{k'}^{D})^{2}.
\]
Thus
\[
\left|\fdimC(q_{k}A_{i_{k}}q_{k})-\fdimC(A_{i_{k}})\right|<\left|\sum_{k'>k}(t_{k'}^{D})^{2}\right|=\fdimC((1-q_{k})D(1-q_{k}))|.
\]
Thus if $\fdimC(D)\ne -\infty$ then $\lim_{k\to\infty}\fdimC(q_{k}A_{i_{k}}q_{k})=\fdimC(A)$.  The same calculation works for $B_{j_{k}}$.  Then note by Theorem \ref{maintheorem}
\[
\fdimC(\mcm(i_{k},j_{k},k))=
\fdimC(q_{k}A_{i_{k}}q_{k})+\fdimC(q_{k}B_{j_{k}}q_{k})
-\fdimC(q_{k}Dq_{k}).
\]
Then if $\fdimC(D)\ne-\infty$ then 
\[
\lim_{k\to\infty}\fdimC(\mcm(i_{k},j_{k},k))=\lim_{k\to\infty}\fdimC(A_{i_{k}})+\lim_{k\to\infty}\fdimC(B_{i_{k}})-\fdimC(D)
\]
\[
=\fdimC(A)+\fdimC(B)-\fdimC(D).
\]
If $\fdimC(D)=-\infty$ then we see that if $\fdimC(A)$ and $\fdimC(B)$ are not $-\infty$, so $\fdimC(q_{k}A_{i_{k}}q_{k})+\fdimC(q_{k}B_{j_{k}}q_{k})$ is bounded below, while $-\fdimC(q_{k}Dq_{k})$ goes to $\infty$, the equation again holds.

\end{proof}

\begin{eg}Let:
\[
D=\bigoplus_{i=1}^{\infty}\underset{\frac{1}{i}}{\overset{p_{i}^{D}}{\bbc}}
\]
\[
A=\bigoplus_{i=1}^{\infty}\overset{p^{A}_{i}}{R}, \tau(p^{A}_{i})=\frac{1}{2i-1}+\frac{1}{2i}
\]
\[
B=\overset{p^{B}_{0}}{R}\oplus\bigoplus_{i=1}^{\infty}\overset{p^{B}_{i}}{R}, \tau(p^{B}_{i})=\frac{1}{2i}+\frac{1}{2i+1}, \tau(p_{0}^{B})=1.
\]
Here we set $p_{2i-1}^{D}+p_{2i}^{D}=p_{i}^{B}$ and $p_{2i}^{D}+p_{2i+1}^{D}=p_{i}^{B}$, with $p_{0}^{B}=p_{1}^{D}$. We can see that $r=\fdimC(A)+\fdimC(B)-\fdimC(D)=\frac{\pi^{2}}{6}$.  From there it is not hard to check that  $A*_{D}B=\mcf_{\frac{\pi^{2}}{6}}^{\infty}$.  This shows that even if all the factor summands of $A$ and $B$ are finite, the amalgamated free product may have a non-finite factor summand (unlike the $\mcr_{1}$ and $\mcr_{2}$ cases).

\end{eg}

\begin{eg}  Let
\[
D=\bigoplus_{i=1}^{\infty}\underset{1}{\bbc}
\]
\[
A=B=\bigoplus_{i=1}^{\infty}\underset{1/2}{\bbc}.
\]
Then computing we see
\[
\fdimC(A)+\fdimC(B)-\fdimC(D)=-\infty-\infty+\infty,
\]
which is not defined.  Computing the product we see that $A*_{D}B=\bigoplus_{i=1}^{\infty}L(\bbz)\otimes M_{2}$.  In this case $\fdimC(A*_{D}B)=0$, however it would be easy to add a summand so that we could get any value from this.

\end{eg}

\section{Closedness of $\mcr_{3}$ and $\mcr_{4}$}

The following lemma was proved as Theorem 3.1 in \cite{KenNate2004}.

\begin{lemma}\label{L:sillyRlemma}  Let $R$ and $R'$ be hyperfinite II$_{1}$ factors. Let $D=\oplus_{i\in I}\overset{p_{i}}{\underset{t_{i}}{\bbc}}$ ($I$ countable) be a common subalgebra of $R$ and $R'$.  Then $R*_{D}R'=L(F_{1+\sum_{i\in I}t_{i}^{2}})$, and this is generated by $R'\cup\{q_{j}X_{j}q_{j}\}_{j\in J}$, where the $X_{j}$ form a semicircular system which is free with $R'$ and $\sum_{j\in J}\tau(q_{j})^{2}=\sum_{i\in I}\tau(p_{i})^{2}$, and $\tau(q_{j})\leq 1$ for all $j\in J$.
\end{lemma}
\begin{proof}  
The fact that $R*_{D}R'$ is the interpolated free group factor mentioned follows from Theorem \ref{T:maininftheorem}, so here we need only check that it is generated correctly.

First assume that $I$ is finite, and proceed by induction on the size of $I$.  In the case $|I|=1$, Corollary 3.6 in \cite{kenIFGF} establishes that $R*R'$ is isomorphic to $R'*L(\bbz)$ with a map preserving $R'$.  If $\tau(p_{1})>1$ then choose some integer $N$ larger than $\tau(p_{1})$, and write this as the algebra generated by $R'\cup\{q_{j}X_{j}q_{j}\}_{j=1}^{n^{2}}$ where the $q_{j}$ are of trace $\frac{1}{n}$. This completes the base case.

 Let $i_{0}$ be such that $t_{i_{0}}$ is minimal.  Using Theorem 2.1 (b) from \cite{kenLMS}, we define 
\[
\mcn_{1}=vN\left((1-p_{i_{0}})R(1-p_{i_{0}})\cup (1-p_{i_{0}})R'(1-p_{i_{0}})\right)
\]
\[
=(1-p_{i_{0}})R(1-p_{i_{0}})*_{(1-p_{i_{0}})D} (1-p_{i_{0}})R'(1-p_{i_{0}}).
\]
Then by our induction hypothesis, $\mcn_{1}=L(F_{1+\sum_{i\in I,i\ne i_{0}}\frac{t_{i}^{2}}{(1-t_{i_{0}})^{2}}})$, generated by  $(1-p_{i_{0}})R'(1-p_{i_{0}})\cup\{q_{j}X_{j}q_{j}\}_{J'\}}$, where $\sum_{j\in J'}\tau(q_{j})^{2}=\sum_{i\in I\backslash\{i_{0}\}}\tau(p_{i})^{2}$. 

Then using part (c) of Theorem 2.1 in \cite{kenLMS}, we see that $r=p_{i_{0}}$, thus $\tilde{A}=(1-p_{i_{0}})R(1-p_{i_{0}})\oplus \bbc p_{i_{0}}$, and thus $p_{i_{0}}\tilde{A}p_{i_{0}}=\bbc$. From this we see directly that for  $\mcn_{2}=vN(\tilde{A}\cup R')$, $\mcn_{1}=(1-p_{i_{0}})\mcn_{2}(1-p_{i_{0}})$.  Thus we write $\mcn_{2}=vN(R'\cup\{q_{j}X_{h}q_{j}\}_{j\in J'})$.

Using part (d) of Theorem 2.1 in \cite{kenLMS}, we see that $p_{i_{0}}(R*_{D}R')p_{i_{0}}=p_{i_{0}}\mcn_{2}p_{i_{0}}*L(\bbz)$.  Then, after splitting up $L(\bbz)$ if $\tau(p_{0})>1$ as before,we see that $R*_{D}R'$ is generated by $R'\cup\{q_{j}X_{j}q_{j}\}_{j\in J}$, completing the proof for $I$ finite.

For the infinite case, let $I=\bbn$, ordered so that $t_{i}\leq t_{i-1}$, and let $r_{k}=\sum_{i=1}^{k}r_{i}^{D}$.  Let $\mcm_{k}=r_{k}Rr_{k}*_{r_{k}D}r_{k}R'r_{k}$.  The above tells us that each $\mcm_{k}$ is generated by $r_{k}R'r_{k}\cup \{q_{j}X_{j}q_{j}\}_{j\in J_{k}}$ and that the embeddings of $\mcm_{k}\to\mcm_{k+1}$ are substandard, thus the generating set in the inductive limit $\mcm$ is as desired, completing the proof.

\end{proof}
\begin{cor}\label{C:sillyRcor} Let $A$ and $B$ each be copies of the hyperfinite II$_{\infty}$ factor, with subalgebra $D=\bigoplus\limits_{i=1}^{\infty}\overset{p_{i}}{\underset{t_{i}}{\bbc}}$ and trace preserving conditional expectation onto $D$ so the trace is semifinite on $D$.  Then $A*_{D}B=\mcf_{-\fdimC(D)}^{\infty}$, and this is generated by $A\cup\{q_{j}X_{j}q_{j}\}_{i=1}^{\infty}$ where the $X_{j}$ form a semicircular system which is free with $A$  and $\sum_{j\in J}\tau(q_{j})^{2}=\sum_{i\in I}\tau(p_{i})^{2}$, and $\tau(q_{j})\leq 1$ for all $j\in J$.
\end{cor}
\begin{proof}
Let $r_{k}=\sum_{i=1}^{k}p_{i}$.  Then by Lemma \ref{L:sillyRlemma} any $\mcn_{k}=r_{k}Ar_{k}*_{r_{k}Dr_{k}}r_{k}Br_{k}$ is generated $r_{k}Ar_{k}\cup\{q_{j}X_{j}q_{j}\}_{j\in J_{k}}$, and furthermore from the proof of that lemma $\mcn_{k-1}\to\mcn_{k}$ is a substandard embedding mapping $r_{k-1}Ar_{k-1}$ into $r_{k}Ar_{k}$.   Since each embedding is substandard, the inductive limit of $\mcn_{k}\to \mcn_{k+1}\to...$ is the desired semifinite interpolated free group factor, and $A$ is preserved.
\end{proof}
\begin{lemma}\label{L:sillyLFlemma}  Let $A=\mcf_{s_{A}}^{t}$ and $B=\mcf_{s_{B}}^{t}$, each with subalgebra $D=\bigoplus\limits_{k\in K}\overset{p_{k}^{D}}{\underset{t_{k}}{\bbc}}$.  Then $\mcm=A*_{D}B\cong\mcf_{s_{A}+s_{B}-\fdimC(D)}^{t}$, and the inclusion of $A\to \mcm$ is substandard.

\end{lemma}
\begin{proof}
Let $p$ be a minimal projection in $D$, and choose representations of $A$ and $B$ of the form $A=vN(R\cup\{p_{i}X_{i}p_{i}\}_{i\in I})$ and $B=vN(R'\cup\{q_{j}X_{j}q_{j}\}_{j\in J})$, where $p_{i}\leq p$ and $q_{j}\leq p$ for all $i\in I,j\in J$, and $R$ and $R'$ are either the hyperfinite II$_{1}$ factor or the hyperfinite II$_{\infty}$ factor.  Then $A*_{D}B$ is generated by $R\cup R'\cup\{p_{i}X_{i}p_{i}\}_{i\in I}\cup\{q_{j}X_{j}q_{j}\}_{j\in J}$.  From Lemma \ref{L:sillyRlemma} and Corollary \ref{C:sillyRcor}  we know that $R*_{D}R'=\mcf_{-\fdim(D)}^{t}$, with $R$ embedded correctly and so $R*_{D}R'$ is generated by $R\cup \{p_{k}^{D}X_{k}p_{k}^{D}\}_{k\in K}$.  We can then find $q'_{j}\in R$ which are unitary conjugates of $q_{j}$ using unitaries in $R*_{D}R'$, thus $\mcm$ is generated by $R\cup\{p_{k}^{D}X_{k}p_{k}^{D}\}_{k\in K}\cup\{p_{i}X_{i}p_{i}\}_{i\in I}\cup \{q_{j}X_{j}q_{j}\}_{j\in J}$.  Thus $\mcm$ is the desired factor, and the inclusion of $A$ is substandard.
\end{proof}

\begin{lemma}\label{L:sillyLFHlemma}  Let $H$ be a hyperfinite von Neumann algebra and $A=\mcf_{s}^{t}$, with common abelian type I atomic subalgebra $D$, then $A*_{D}H=\mcf_{s+\fdimC(H)-\fdimC(D)}^{t}$, and the inclusion of $\mcf_{s}^{t}$ into it is substandard.
\end{lemma}
\begin{proof}  Using Lemma \ref{L:infchainlemma} we can find a chain of multimatrix algebras $H_{i}$ approximating $H$ so that $D\subseteq H_{i}$ for all $i$, and that so that for each $i$ there exists a finite projection $q_{i}\in D$ so that $(1-q_{i})H_{i}(1-q_{i})=(1-q_{i})D$.  Applying Lemma \ref{L:decomp2} we can assume that this is generated by a sequence of simple steps.  Then by Lemma 5.9 of \cite{kenAmJM}, since $q_{i}Aq_{i}$ is an interpolated free group factor, $q_{i}H_{i}q_{i}*_{q_{i}D}q_{i}Aq_{i}$ is an interpolated free group factor (with regulated dimension as expected), and the inclusion of $q_{i}Aq_{i}$ is standard.  Then apply Lemma \ref{L:removesummand} to see this implies $H_{i}*_{D}A$ is $\mcf_{s+\fdimC(H_{i})-\fdimC(D)}^{t}$, and the inclusion of $A$ into it is substandard.

Using Lemma \ref{tensorlemma1}, and Lemma \ref{gluelemma}  we see that the embedding of $H_{i}*_{D}A\to H_{i+1}*_{D}A$ is substandard, which completes the proof.

\end{proof}
\begin{lemma}\label{L:sillyLFRlemma}  Let $A\in \mathcal{R}_{4}$ and $B=\mcf_{s}^{t}$.  Let $D$ be an abelian type I atomic subalgebra of both $A$ and $B$.  Then $\mcm=A*_{D}B$ is a semifinite interpolated free group factor, and if $\rdim(A)$ is defined and greater than $-\infty$, then $\mcm=\mcf_{s+\fdimC(A)-\fdimC(D)}^{t}$. Furthermore the inclusion $B\to \mcm$ is substandard and that of $F_{i}\to\mcm$ is substandard for any  summand, $F_{i}\cong\mcf_{s'}^{t'}$, of $A$.
\end{lemma}
\begin{proof}  Let $A=\overset{p_{0}}{H}\oplus\bigoplus\limits_{i=1}^{K}\overset{p_{i}}{\mcf_{s_{i}}^{t_{i}}}$ for $K\in \bbn\cup\{\infty\}$ and where $H$ is a hyperfinite algebra.  Let $A_{0}=H\oplus\bigoplus\limits_{i=1}^{K}p_{i}D$ and let $A_{k}=H\oplus\bigoplus\limits_{i=1}^{k}\mcf_{s_{i}}^{t_{i}}\oplus\bigoplus\limits_{i=k+1}^{K}p_{i}D$.  Let $\mcm_{k}=A_{k}*_{D}B$.

Our goal is to show that $\mcm_{k}=\mcf_{\fdimC(A_{k})+s-\fdimC(D)}^{t}$ for all $k$, and that the inclusion of $\mcm_{k}\to \mcm_{k+1}$ is substandard.

First, Lemma \ref{L:sillyLFHlemma} shows that $\mcm_{0}=\mcf_{\fdimC(A_{0})+s-\fdimC(D)}^{t}$, and the inclusion of $B\to \mcm_{0}$ is standard.


Lemma \ref{L:removesummand} tells us that $p_{k}\mcm_{k}p_{k}=p_{k}\mcm_{k-1}p_{k}*_{p_{k}D}\mcf_{s_{k}}^{t_{k}}$ and that $\mcm_{k}$ is a factor. Thus the embedding of $\mcf_{s_{k}}^{t_{k}}\to p_{k}\mcm_{k}p_{k}$ is substandard. By Lemma \ref{L:sillyLFlemma}, $p_{k}\mcm_{k}p_{k}=\mcf_{s'}^{t_{k}}$ for some $s'$ and the inclusion of $p_{k}\mcm_{k-1}p_{k}$ into it is substandard, and thus $\mcm_{k}=\mcf_{\fdimC(A_{k})+s-\fdimC(D)}^{t}$ and the inclusion $\mcm_{k-1}\to\mcm_{k}$ is substandard.

Thus the inductive limit of the $\mcm_{k}$, $A*_{D}B=\mcf_{\fdimC(A)+s-\fdimC(D)}^{t}$, the inclusion of $B$ into it is standard, and the inclusions of $\mcf_{s_{i}}^{t_{i}}\to \mcm$ are substandard.
\end{proof}

\begin{theorem}  For $A,B\in \mathcal{R}_{4}$ (or $\mcr_{3}$) and $D$ an atomic type I common subalgebra of $A$ and $B$, $A*_{D}B\in \mathcal{R}_{4}$ (or $\mcr_{3}$), and $\fdimC(A*_{D}B)=\fdimC(A)+\fdimC(B)-\fdimC(D)$ if this equation is well defined.  
\end{theorem}
\begin{proof}  As usual we use Lemma \ref{L:abelianD} to assume without loss of generality that $D$ is abelian.  Let $A=H_{A}\oplus\bigoplus\limits_{i=1}^{K_{A}}\overset{p_{i}}{\mcf_{s_{i}}^{t_{i}}}$ where $K_{A}\in\bbn\cup\{\infty\}$ similarly let $B=H_{B}\oplus\bigoplus\limits_{j=1}^{K_{B}}\overset{q_{j}}{\mcf_{s_{j}'}^{t_{j}'}}$.  Then define $A_{k}=H_{A}\oplus\bigoplus\limits_{i=1}^{k}\mcf_{s_{i}}^{t_{j}}\oplus\bigoplus\limits_{i=k+1}^{K_{A}}p_{i}D$ and define $B_{k}$ similarly.

Let $\mcm(i,j)=A_{i}*_{D}B_{j}$, then the inductive limit of the $\mcm(i,j)$ is $A*_{D}B$.  We claim that each $\mcm(i,j)$ is in $\mathcal{R}_{4}$, and that the embedding of $\mcm(i,j)$ into $\mcm(i+1,j)$ is substandard.  We proceed by induction on $i,j$.

Theorem \ref{T:maininftheorem} tells us that $\mcm(0,0)$ is in $\mathcal{R}_{4}$, and that $\fdimC(\mcm(0,0))=\fdimC(A_{0})+\fdimC(B_{0})-\fdimC(D)$.  By Lemma \ref{L:removesummand},  $p_{i}\mcm(i,j)p_{i}=vN(p_{i}\mcm(i-1,j)p_{i}\cup \mcf_{s_{i}}^{t_{i}})=p_{i}\mcm(i-1,j)p_{i}*_{p_{i}D}\mcf_{s_{i}}^{t_{i}}$.  Lemma \ref{L:sillyLFRlemma} tells us that this is of the form $\mcf_{s}^{t}$ for some $s$ and $t$ and that the inclusion of any $\mcf_{s'}^{t}$ factor summand in $p_{i}\mcm(i-1,j)p_{i}\to p_{i}\mcm(i,j)p_{i}$ is substandard.  This also tells us that the inclusion of $\mcf_{s_{i}}^{t_{i}}\to p_{i}\mcm(i,j)p_{i}$ is substandard.


Since the central support of $p_{i}$ in $\mcm(i-1,j)$ and $\mcm(i,j)$ is the same (again by Lemma \ref{L:removesummand}), any $F_{s}^{t}$ factor summand of $\mcm(i-1,j)$ which is orthogonal to $p_{i}$ is identical in $\mcm(i,j)$.  Thus the inclusion of $\mcm(i-1,j)\to \mcm(i,j)$ is substandard, as is the inclusion of $\mcf_{s_{i}}^{t_{i}}\to\mcm(i-1,j)$.

Next note that $\fdimC(q\mcm(i,j)q)=\fdimC(p_{i}\mcm(i,j)p_{i})$ by Lemma \ref{L:rdimP}.   Note 
\[
\fdimC(p_{i}\mcm(i,j)p_{i})=\fdimC(p_{i}\mcm(i-1,j)p_{i})+\fdim(\mcf_{s_{i}}^{t_{i}})-\fdimC(p_{i}D).
\]
Since $q$ is the central support of $p_{i}$ in both $\mcm(i,j)$ and $\mcm(i-1,j)$ we know 
\[
\fdimC (p_{i}\mcm(i-1,j)p_{i})=\fdimC(q\mcm(i-1,j)q).
\]  
Thus we see
\[
\fdimC(q\mcm(i,j)q)=\fdimC(q\mcm(i-1,j)q)+\fdimC(\mcf_{s_{i}}^{t_{i}})-\fdimC(p_{i}D).
\]
And since $(1-q)\mcm(i,j)(1-q)=(1-q)\mcm(i-1,j)(1-q)$, we see
\[
\fdimC(\mcm(i,j))=\fdimC(\mcm(i-1,j))+\fdimC(\mcf_{s_{i}}^{t_{i}})-\fdimC(p_{i}D).
\]
Note $\fdim(p_{i}A_{i}p_{i})=\fdimC(\mcf_{s_{i}}^{t_{i}})$,
and $\fdimC(p_{i}A_{i-1}p_{i})=\fdimC(p_{i}D)$. Thus 
\[\fdimC(A_{i})-\fdimC(A_{i-1})=\fdimC(\mcf_{s_{i}}^{t_{i}})-\fdimC(p_{i}D).
\]
Combining these we get:
\begin{multline*}
\fdimC(\mcm(i,j))\\=\fdimC(A_{i-1})+\fdimC(B_{j})-\fdimC(D)+\fdimC(\mcf_{s_{i}}^{t_{i}})-\fdimC(p_{i}D)
\end{multline*}
\[
=\fdimC(A_{i})+\fdimC(B_{j})-\fdimC(D).
\]

Thus our claim is proved, and the result follows.  Note that if $A$ and $B$ are finite, their amalgamated free product is, and thus if they were both in $\mcr_{3}$ then heir amalgamated free product is in $\mcr_{3}$.
\end{proof}

\begin{eg}  Let $A=\bigoplus_{i=1}^{\infty}\mcf^{2^{-i}}_{2^{i}}$ where $\tau(p_{i})=\frac{1}{2^{i}}$ and let $B=\mcf^{1}_{2}$, and choose any multimatrix subalgebra $D$ of these.  Checking, we see $\rdim(A)=\sum_{i=1}^{\infty}1=\infty$.  By Lemma \ref{L:sillyLFRlemma} we know $A*_{D}B$ is an interpolated free group factor, and applying the free dimension formula we see that it must be $\mcf_{\infty}^{1}$. 
\end{eg}
 This shows that unlike in the case of $\mcr_{1}$ and $\mcr_{2}$, restricting the interpolated free group factor summands in $A$ and $B$ to having only finite free dimension does not guarantee this for the product.  We could restrict $\mcr_{3}$ to those with finite regulated dimension (or equivalently free dimension) and it would be still closed under these amalgamated free products.  Similarly we could restrict $\mcr_{4}$ to those with finite defined regulated dimension, which would then be closed over amalgamated free products of type I atomic algebras with finite regulated dimension.

\bibliography{newbib}

\end{document}